\documentclass[12pt,a4paper]{article}

\usepackage{mathptmx}       
\usepackage{helvet}         
\usepackage{courier}        

\usepackage{graphicx}        

\usepackage[a4paper, hmargin={2.5cm,2.5cm},vmargin={3.3cm,3.3cm}]{geometry}
\usepackage{amsmath, amssymb, amsfonts, amsbsy, latexsym, color}
\usepackage{amscd,amsxtra}
\usepackage{amsthm}
\usepackage[normalem]{ulem}

\usepackage[T1]{fontenc}
\usepackage[latin1]{inputenc}
\usepackage[english]{babel}
\usepackage{cite}
\usepackage{paralist}
\usepackage{url}
\usepackage[sf,bf,compact,pagestyles,small]{titlesec}
\titlespacing*{\section}{0pt}{14pt}{4pt}
\titlespacing*{\subsection}{0pt}{8pt}{3pt}
\usepackage{mathdots}

\usepackage{fancyhdr,lastpage}
\def\maketimestamp{\count255=\time
\divide\count255 by 60\relax
\edef\thetime{\the\count255:}%
\multiply\count255 by-60\relax
\advance\count255 by\time
\edef\thetime{\thetime\ifnum\count255<10 0\fi\the\count255}
\edef\thedate{\number\day-\ifcase\month\or Jan\or Feb\or Mar\or
             Apr\or May\or Jun\or Jul\or Aug\or Sep\or Oct\or
             Nov\or Dec\fi-\number\year}
\def\timstamp{\hbox to\hsize{\tt\hfil\thedate\hfil\thetime\hfil}}}
\maketimestamp
\fancypagestyle{plain}{%
\fancyhf{} 
\fancyfoot[C]{\small \thepage{} of \pageref{LastPage}}}

\numberwithin{equation}{section}  
\allowdisplaybreaks[4]

\newtheorem{theorem}{Theorem}[section]

\newtheorem{lemma}[theorem]{Lemma}
\newtheorem{proposition}[theorem]{Proposition}
\newtheorem{corollary}[theorem]{Corollary}

\theoremstyle{definition}
\newtheorem{definition}[theorem]{Definition} 

\newtheorem{example}{Example}
\theoremstyle{remark}
\newtheorem{remark}{Remark}

\DeclareMathOperator{\Span}{span} %
\DeclareMathOperator{\diagonal}{diag} %

\DeclareMathOperator{\exponential}{e}

\newcommand{\myexp}[1]{\exponential^{#1}}
\newcommand{\diag}[1]{\diagonal{(#1)}}

\newcommand{\id}[1][]{I_{#1}} 

\newcommand{\card}[1]{\# {#1}}

\newcommand{\union}{\cup}

\newcommand*{\numbersys}[1]{\ensuremath{\mathbb{#1}}}
\newcommand*{\C}{\numbersys{C}}
\newcommand*{\R}{\numbersys{R}}

\newcommand*{\Z}{\numbersys{Z}}

\newcommand*{\N}{\numbersys{N}}
\newcommand*{\scalarfield}[1]{\ensuremath{\mathbb{#1}}}
\newcommand*{\K}{\scalarfield{K}}

\newcommand{\itvco}[2]{\ensuremath{\left[{#1},{#2}\right)}} %
\newcommand{\abs}[1]{\ensuremath{\left\lvert#1\right\rvert}}
\newcommand{\abssmall}[1]{\ensuremath{\lvert#1\rvert}}

\newcommand{\norm}[2][]{\ensuremath{\left\lVert#2\right\rVert_{#1}}}

\newcommand{\innerprod}[3][]{\ensuremath{\left\langle #2,#3\right\rangle_{\! #1}}}

\newcommand{\innerprodbig}[3][]{\ensuremath{\bigl \langle #2,#3\bigr\rangle_{\!\!#1}}}
\newcommand{\innerprods}[2]{\ensuremath{\langle #1,#2\rangle}}

\newcommand{\set}[1]{\ensuremath{\left\lbrace{#1}\right\rbrace}}
\newcommand{\setprop}[2]{\ensuremath{\left\lbrace{#1} : {#2}\right\rbrace}}

\newcommand{\setsmall}[1]{\ensuremath{\lbrace{#1}\rbrace}}
\newcommand{\setpropbig}[2]{\ensuremath{\bigl\lbrace{#1} :
      {#2}\bigr\rbrace}}
\newcommand{\setpropsmall}[2]{\ensuremath{\lbrace{#1} : {#2}\rbrace}}

\newcommand{\floor}[1]{\left\lfloor #1 \right\rfloor}

\makeatletter
\renewcommand*\env@matrix[1][c]{\hskip -\arraycolsep
  \let\@ifnextchar\new@ifnextchar
  \array{*\c@MaxMatrixCols #1}}
\makeatother
\newlength{\bracewidth}
\newcommand{\myunderbrace}[2]{\settowidth{\bracewidth}{$#1$}#1\hspace*{-1\bracewidth}\smash{\underbrace{\makebox{\phantom{$#1$}}}_{#2}}}

\usepackage{xspace}
\newcommand{\ie}{i.e.,\xspace} 
\newcommand{\eg}{e.g.,\xspace} 

\DeclareMathOperator{\HTF}{HTF}
\DeclareMathOperator{\STF}{STF}

\usepackage{hyperref}
\hypersetup{
 pdfview={FitH},
 pdfstartview={FitH},
 bookmarks={true},
 pdftoolbar={false},
 pdfmenubar={false},
 pdfauthor = {Jakob Lemvig, Christopher Miller, and Kasso A. Okoudjou},
 pdftitle = {Prime tight frames},
 pdfsubject = {finite frame theory},
 pdfkeywords = {divisible frames, equiangular tight frames, frames, harmonic tight frames,  prime frames, spectral tetris frames, tight frames},
 pdfcreator = {LaTeX with hyperref package},
 pdfproducer = {PDFlatex}}

\makeatletter
\def\blfootnote{\xdef\@thefnmark{}\@footnotetext} 
\def\subjclass{\xdef\@thefnmark{}\@footnotetext}
\long\def\symbolfootnote[#1]#2{\begingroup%
\def\thefootnote{\fnsymbol{footnote}}\footnote[#1]{#2}\endgroup} 
\if@titlepage
  \renewenvironment{abstract}{%
      \titlepage
      \null\vfil
      \@beginparpenalty\@lowpenalty
      \begin{center}%
        \bfseries \abstractname
        \@endparpenalty\@M
      \end{center}}%
     {\par\vfil\null\endtitlepage}
\else
  \renewenvironment{abstract}{%
      \if@twocolumn
        \section*{\abstractname}%
      \else
        \small
        \list{}{%
          \settowidth{\labelwidth}{\textbf{\abstractname:}}
          \setlength{\leftmargin}{50pt}
          \setlength{\rightmargin}{50pt}
          \setlength{\itemindent}{\labelwidth}
          \addtolength{\itemindent}{\labelsep}
        }
        \item[\textbf{\abstractname:}]

      \fi}
      {\if@twocolumn\else\endlist\fi}
\fi
\makeatother

\begin{document}

\title{Prime  tight frames}

\author{Jakob Lemvig\footnote{Technical University of Denmark, Department of Mathematics, Matematiktorvet 303, 2800 Kgs. Lyngby, Denmark, E-mail: \url{J.Lemvig@mat.dtu.dk}}\phantom{$\ast$}, Christopher Miller\footnote{8803 Pisgah Drive Clinton MD, 20735 USA, E-mail: \url{christopherdmiller5@gmail.com}}, and Kasso A.~Okoudjou\footnote{Department of Mathematics, University of Maryland, College Park, MD, 20742 USA, E-mail: \url{kasso@math.umd.edu}}}

\date{\today}

\maketitle 

\thispagestyle{plain}
\begin{abstract} We introduce a class of finite tight frames called prime tight frames and prove some of their elementary properties. In particular, we show that any finite tight frame can be written as a union of prime tight frames.  We then characterize all prime harmonic tight frames and use this characterization to suggest effective analysis and synthesis computation strategies for such frames.  Finally, we describe all prime frames constructed from the spectral tetris method, and, as a byproduct, we obtain a characterization of when the spectral tetris construction works for redundancies below two. 
\end{abstract}

\section{Introduction}
\label{sec:introduction}
A frame for a finite dimensional vector space is a spanning set that is not necessarily a basis. More
specifically, let $\K$ denote either $\R$ or $\C$, and consider $\K^N$, $N\geq 1$, as a vector space over the
scalar field $\K$.  Given $M\geq N$, a collection of vectors $\Phi=\{\varphi_i\}_{i=1}^M\subset\K^N$ is called
a \emph{finite frame} for $\K^N$ if there are two constants $0<A\le B$ such that
\begin{equation}\label{eq:frameeq}
A\|x\|^2 \le \sum_{i=1}^M |\innerprod{x}{\varphi_i}|^2 \le B\|x\|^2 ,\quad\text{for all $x\in\K^N$.}
\end{equation}
If the frame bounds $A$ and $B$ are equal, the frame $\Phi=\{\varphi_i\}_{i=1}^M$ is called a \emph{finite
  tight frame}. We refer to $\Phi=\{\varphi_i\}_{i=1}^M$ as a \emph{finite unit norm tight frame} (FUNTF), if
$\Phi$ is a tight frame with $\|\varphi_i\|=1$ for each $i$. In this case, the frame bound is $A=M/N$. For a
tight frame $\Phi=\{\varphi_i\}_{i=1}^M$ with frame bound $A$, the following reproducing formula holds
\begin{equation}\label{eq:reconstruction_for_frames}
x =  \frac1A \sum_{i=1}^M  \innerprod{x}{\varphi_i}\varphi_i,\quad \text{for all $x\in\K^N$.}
\end{equation}

The decomposition formula provided by frames has important consequences in many areas of science and
engineering in which they now play an increasingly important role. We refer to \cite{chris03, koche1, koche2}
for an overview of frames and some of their applications.  In particular, tight frames and FUNTFs have
attracted a lot of attention in recent years due to their numerous applications.  In this context methods for
characterizing and constructing these types of frames are being actively investigated. For instance, the
existence and characterization of FUNTFs was settled by Benedetto and Fickus in \cite{bf03} as minimizers of a
potential function.

In this paper we are interested in the classification of the tight frames that remain tight after the deletion
of frame vectors. For (non-tight) frames, this problem is known as the erasure problem for frames and was
first investigated by Goyal, Kova\v{c}evi\'{c}, and Kelner \cite{Goyal:2001aa}, and later by Casazza and
Kova\v{c}evi\'{c} \cite{cako03}. The focus in these works was whether any set of vectors of a given size can
be removed from a frame to still leave a frame. In this case, estimates for frame bounds after erasures were
given.

By contrast, we investigate tight frames which remain such after the erasures of \emph{some} set of frame
vectors. In the process, we define a new class of tight frames called \emph{prime tight frames} as tight
frames for which no proper sub-collection is a tight frame.  In Section~\ref{sec:def}, we analyze the
structure of prime tight frames. In particular, we show that prime tight frames are fundamental building
blocks for all tight frames in the sense that any tight frame can be written as a union of prime tight
frames. This is, in a way, analog to the prime factorization of natural numbers; however, this factorization
of tight frames into prime ones is not unique.  In Section~\ref{sec:special-classes}, we then restrict our
attention to FUNTFs and characterize the prime tight frames for three families of FUNTFs: equiangular tight
frames, harmonic tight frames, and spectral tetris frames. Equiangular and harmonic tight frames have proven
to be some of the most useful frames in a variety of applications.  For harmonic tight frames (HTF), \ie tight
frames constructed from an $M \times M$ discrete Fourier matrix by keeping the first $N$ rows, this
characterization leads to effective analysis and synthesis computation strategies for HTFs.  The spectral
tetris method was recently introduced by Casazza, Fickus, Mixon, Wang, and Zhou~\cite{cfmwz11} and has already
received considerable attention in the frame theory community. As a byproduct of our results on prime spectral
tetris frames, we are able to characterize all dimensions $N$ and number of frame vectors $M$ for which the
spectral tetris construction works. This was previously only known for redundancies larger than two.

\section{Basic facts of prime and divisible tight frames }
\label{sec:def}
We define prime and divisible tight frames and prove some of their properties in
Section~\ref{subsec:def-prop}. In Section~\ref{sec:prime-exist-dense}, we then prove that prime frames exist
in every dimension which allows us to conclude that prime tight frames form an open, dense subset of the set
of all tight frames.

\subsection{Definitions and elementary properties}\label{subsec:def-prop}

\begin{definition}\label{def:ptf}
  Let $M\geq N$ be given. A tight frame $\Phi=\{\varphi_i\}_{i=1}^{M}$ in $\K^N$ is said to be \emph{prime},
  if no proper subset of $\Phi$ is a tight frame for $\K^N$. If $\Phi$ is not prime, we say that it is
  \emph{divisible}. In particular, given an integer $p$ with $N\leq p \leq M-N$, the tight frame $\Phi$ is
  $(M, p)$-divisible if there exists a subset of $\Phi$ containing $p$ vectors that is also a tight frame.
\end{definition}

\begin{remark}\label{lem:atleast-2-times-dimension}
  For a tight frame $\Phi=\{\varphi_i\}_{i=1}^{M}$ in $\K^N$ to be prime it is sufficient that $M< 2N$. In
  other words, for $\Phi$ to be divisible it is necessary that $M\geq 2N$. This follows from the fact that, if
  $M < 2N$, it is impossible to partition $\Phi$ into two spanning sets.
\end{remark}

We also note that there exist frames $\Phi$ that are $(M, p)$-divisible for all $p$ in the full range $N \le p
\le M-N$, \eg choosing $\Phi$ to be the harmonic tight frame in $\C^2$ with $24$ frame vectors. On the other
hand, it is not possible for a tight frame to be robust with respect to tightness for any $p$ erasures, where
$p \in \N$ is fixed. Hence within the class of tight frames the erasure problem has no solution.  This
negative result and a simple symmetry observation are stated in the next proposition.
  
\begin{proposition}
  Let $M\geq N\geq 2$ be given. Let $\Phi=\{\varphi_i\}_{i=1}^{M}$ be a tight frame in $\K^N$.
\begin{compactenum}[(i)]
\item $\Phi$ is $(M, p)$-divisible if and only if $\Phi$ is $(M, N-p)$-divisible.
\item Given $N\leq p \leq M-N$, there exists no $(M, p)$-divisible tight frame $\Phi$ in $\K^N$ for which
  every sub-collection of $\Phi$ consisting of $p$ vectors is a tight frame.
\end{compactenum}    
\end{proposition}  
  
\begin{proof} 
  (i): The proof of this part is trivial and so we omit it. \\
  (ii): First assume that $\Phi=\{\varphi_k\}_{k=1}^{M}$ is a FUNTF that is $(M, p)$-divisible and such that
  every sub-collection of $p$ vectors is again a FUNTF. Let $\Phi_1=\{\varphi_i\}_{i \in J}$, where $J =
  \set{1,\dots,p}$, be the tight frame of the first $p$ vectors from $\Phi$.  We replace the $\ell$th frame
  vector in $\Phi_1$ by the $k$th vector in $\Phi$, where $p+1\leq k \leq M$, and denote this tight frame
  $\Phi_2:=\{\varphi_i\}_{i \in J\setminus\set{\ell}} \cup \{\varphi_{k}\}$. Then, given any $x \in \K^N$, we
  have
  \[
  \tfrac{p}{N}\|x\|^2 = \sum_{i \in J}\abs{\innerprod{x}{\varphi_i}}^2 =
  \sum_{i \in J\setminus \set{\ell}}\abs{\innerprod{x}{\varphi_i}}^2 + \abs{\innerprod{x}{\varphi_k}}^2.
  \]
  This implies that $\abs{\innerprod{x}{\varphi_{\ell}}}=\abs{\innerprod{x}{\varphi_k}}$. Since $\ell$ and $k$
  are arbitrary, we have that $\abs{\innerprod{x}{\varphi_\ell}}=\abs{\innerprod{x}{\varphi_k}}$ for each $x
  \in \K^N$, each $\ell =1, 2, \hdots, p$, and each $p+1\leq k\leq M$. In particular, we have that
  $\abs{\innerprod{\varphi_k}{\varphi_\ell}}=1$ for $1 \le \ell \le p$ and $p+1 \le k \le M$. Since the
  vectors $\varphi_i$ are unit-norm, for a fixed $k=k_0 \in \set{p+1,\dots,M}$, this implies that
  $\varphi_\ell \in \Span{(\varphi_{k_0})}$ for all $1 \le \ell \le p$.  Hence, the span of $\Phi_1$ is
  one-dimensional which contradicts the assumption that $\Phi_1$ is a tight frame. We conclude that there
  exist no FUNTF that is $(M, p)$-divisible and for which all sub-collections of $p$ vectors are also
  FUNTF. This proof carries over also to the case of equal norm tight frames.
  
  Now suppose that $\Phi=\{\varphi_{i}\}_{i=1}^{M}$ is an $(M, p)$-divisible non-equal norm, tight frame such
  that any subset of $p$ vectors is again a tight frame. Assume without loss of generality that
  $\|\varphi_1\|\geq \|\varphi_2\|\geq \hdots \ge \|\varphi_M\|.$ Let $\Phi_1=\{\varphi_i\}_{i=1}^{p}$ be the
  subset of the first $p$ vectors from $\Phi$, and let $\Phi_{2}=\{\varphi_i\}_{i=1}^{p-1} \cup
  \{\varphi_{k}\}$ where $p+1\leq k \leq M$, $p$ being the first index such that $\|\varphi_p\|>
  \|\varphi_k\|$. Then, given any $x \in \K^N$, we have
  \[
  \tfrac{\sum_{i=1}^{p}\|\varphi_i\|^2}{N}\|x\|^2 = \sum_{i=1}^{p}\abs{\innerprod{x}{\varphi_i}}^2\qquad
  \text{and} \qquad \tfrac{\sum_{i=1}^{p-1}\|\varphi_i\|^2 + \|\varphi_k\|^2}{N}\|x\|^2 =
  \sum_{i=1}^{p-1}\abs{\innerprod{x}{\varphi_i}}^2 + \abs{\innerprod{x}{\varphi_k}}^2.
  \]
  This implies that $ \abssmall{\innerprods{x}{\varphi_p}}^2 -
  \abs{\innerprod{x}{\varphi_k}}^2=\tfrac{\|\varphi_p\|^2 - \|\varphi_k\|^2}{N} \|x\|^2 $ for each $x \in
  \K^N$. If we choose $x=\varphi_p$, we find that
  \[
  \norm{\varphi_p}^4 - \abssmall{\innerprods{\varphi_p}{\varphi_k}}^2 = \tfrac{\|\varphi_p\|^2 -
    \|\varphi_k\|^2}{N} \|\varphi_p\|^2,
  \]
  and similarly, for $x=\varphi_k$, we have
  \[
  \abssmall{\innerprods{\varphi_p}{\varphi_k}}^2-\|\varphi_k\|^4 =
  \tfrac{\|\varphi_p\|^2 - \|\varphi_k\|^2}{N} \|\varphi_k\|^2.
  \]
  Combining these two equations we obtain
  \[
  \|\varphi_p\|^4 - \|\varphi_k\|^4 = \tfrac{\|\varphi_p\|^4 - \|\varphi_k\|^4}{N},
  \]
  which leads to $N=1$. This contradiction concludes the proof.
\end{proof}

The following result shows that if we can take out a subset of vectors from a tight frame such that these
vectors form a tight frame, then what is left is automatically also a tight
frame. 

\begin{theorem}
\label{thm:splits-into-two-tight-frames}
Let $M\geq 2N$. Suppose $\Phi=\{\varphi_k\}_{k=1}^{M}$ is a divisible tight frame for $\K^N$, and let $\Phi_1
\subsetneq \Phi$ denote a tight frame for $\K^N$. Then $\Phi_2=\Phi \setminus \Phi_1$ is also a tight frame
for $\K^N$.

Moreover, given $M\geq N$, every tight frame of $M$ vectors in $\K^N$ is a finite union of prime tight frames.
\end{theorem}

\begin{proof}
  Assume that $\Phi$ and $\Phi_1 \subsetneq \Phi$ are tight frames with frame bounds $A$ and $A_1$,
  respectively, and let $\Phi_2 = \Phi \setminus \Phi_1$. We will now consider $\Phi$ as an $N \times M $
  matrix of the form $\Phi=[\varphi_i]_{i=1}^M \in \mathrm{Mat}(N \times M, \K)$. Hence, after a possible
  reordering of columns, we have that $\Phi = [\Phi_1\, \Phi_2]$, and
  \begin{equation}
    \label{eq:ortho-split-condition}
    AI_N =\Phi \Phi^* =
    \left[\, \Phi_1 \; \Phi_2 \right] 
    \begin{bmatrix}
      \Phi_1^\ast \\ \Phi_2^\ast
    \end{bmatrix} = \Phi_1 \Phi_1^\ast + \Phi_2 \Phi_2^\ast = A_1 I_N + \Phi_2 \Phi_2^\ast,
  \end{equation}
  which implies that $\Phi_2 \Phi_2^\ast = (A - A_1)\, I_N$. Consequently, $\Phi_2$ is tight frame with frame
  bound $A-A_1>0$.

  For the second part, observe that if $\Phi$ is prime, there is nothing to prove. So suppose that $\Phi$ is
  divisible tight frame with frame bound $A$.  Then, by definition, we can partition $\Phi$ into $\Phi =
  \Phi_1 \cup (\Phi\setminus\Phi_1)$, where $\Phi_1 \subsetneq \Phi$ is a tight frame. It follows from the
  proof of the first part of the lemma that $\Phi_2:=\Phi\setminus \Phi_1$ is also a tight frame.  If $\Phi_1$
  and $\Phi_2$ are prime, we are done. If not, repeat the process. In each step of this procedure we split a
  tight frame into two sets of cardinality at least $N$ each. Hence, by
  Remark~\ref{lem:atleast-2-times-dimension}, the procedure terminates after at most $\floor{\frac{M}{N}}$
  steps.
\end{proof}

The second part of Theorem~\ref{thm:splits-into-two-tight-frames} suggests the following definition. 
\begin{definition}
\label{def:factors}
Let $\Phi$ be a tight frame. Suppose, for some $K\in \N$,
\begin{equation}
  \label{eq:factoring}
 \Phi = \Phi_1 \cup \dots \cup \Phi_K ,
\end{equation}
where $\Phi_k$, $k=1,\dots,K$, are prime tight frames. We shall say that $\Phi_k$ are \emph{prime factors} or
\emph{prime divisors} of $\Phi$.
\end{definition}

The prime factors of a frame are, in general, not unique as shown by the following examples.

\begin{example}
\label{exa:factors-not-unique}
\begin{compactenum}[(a)]
\item The $6$th roots of unity FUNTF in $\R^2$, that is, the frame of six unit-norm vectors each $2\pi/6$
  apart as the vertices of a hexagon, factors into two FUNTFs, each of which consists of three vectors. But
  these are not unique since you can choose the two set of three frame vectors in eight different
  ways. However, up to multiplication of individual frame vectors by $-1$ this decomposition is in fact
  unique. 
\item The uniqueness of the factors up to scalar multiplication in part~(a) of the example does not hold in
  general. The harmonic frame of 10 vectors in $\C^2$ decomposes into either five frames of size two or two
  frames of size five. We refer to Section~\ref{sec:harmonic-frames} for more details on this.
\end{compactenum}
\end{example}

Tight frames $\Phi$ and $\Psi$ are \emph{unitarily equivalent} if there is a bijection $p: \set{1,\dots,M} \to
\set{1,\dots,M}$, a unitary $N \times N$ matrix $U$ and a $c_i \in \K$ with $\abs{c_i}=c>0$ such that $\psi_i
= c_i U \varphi_{p(i)}$ for all $i=1,\dots,M$. In matrix notation this reads $\Psi=U \Phi P C$, where $P$ is
the $M \times M$ permutation matrix for $p$ and $C = \diag{c_1,\dots,c_M}$. Note that there exist other
related equivalence notions, \eg in \cite{Chien:2011} $c_i$ is replaced by a fixed positive scalar $c>0$. We
remark that it is only necessary to introduce permutations in the equivalence relations if one considers
frames as sequences of vectors as opposed to non-ordered collections of vectors with repetition allowed. The
following result shows that prime frame are equivalence classes.

\begin{proposition}
\label{thm:unitary-equi}
Suppose $\Phi$ and $\Psi$ are unitarily equivalent tight frames for $\K^N$.  Then $\Phi$ is prime if and only
if $\Psi$ is prime.
\end{proposition}
\begin{proof}
  We prove that $\Phi$ is divisible if and only if $\Psi$ is divisible.  Assume $\Phi_J := [\varphi_i]_{i \in
    J}$ is a divisor of $\Phi$ for some $J \subset \set{1,\dots,M}$. Then $\Phi_J \Phi_J^{\,\ast}=A I_N$ for
  some $A>0$. Since $\psi_{p^{-1}(i)} = c_{p^{-1}(i)} U \varphi_i$ by assumption, we have that
  \begin{align*}
    \Psi_{p^{-1}(J)} (\Psi_{p^{-1}(J)})^{\ast} = U\Phi_JC_J \, (U\Phi_JC_J )^\ast = c^2 A\, I_N,
  \end{align*}
  where $C_J = \diag{c_{p^{-1}(i)}}_{i \in J}$. This shows that $\Psi$ is divisible with divisor
  $\Psi_{p^{-1}(J)}$. By the symmetry of the equivalence relation, this completes the proof.
\end{proof}

Proposition~\ref{thm:unitary-equi} also holds for the notion of unitarily equivalence used in
\cite{Chien:2011}. However, the result is false if one uses the coarser notion of \emph{equivalence}, where
the matrix $U$ is only assumed to be invertible.

We end this subsection by mentioning that the notion of prime tight frames easily generalizes to non-tight
frames. Recall that if $\Phi$ is a frame with frame operator $S = \Phi \Phi^\ast$, then the associated
canonical Parseval frame is given by $S^{-1/2}\Phi$. Hence, a frame $\Phi$ is said to be \emph{prime} if
its canonical Parseval frame is a prime tight frame.

\subsection{Existence and denseness of prime tight frames}
\label{sec:prime-exist-dense}
We now turn the question of existence of prime tight frames. If $\Phi$ is a union of an orthonormal basis for
$\K^N$ and $M-N$ zero vectors, then $\Phi$ is prime for any $N,M \in \N$. This trivial observation shows the
existence of prime tight frames for all $M\ge N \in \N$. The following result extends this fact to tight
frames of \emph{non-zero} vectors.

\begin{proposition}\label{thm:partial-existence}
  For each dimension $N \in \N$ there exists a prime tight frame for $\K^N$ with $M$ non-zero vectors for any
  $M\ge N \in \N$.
\end{proposition}

\begin{proof}
  The case $N=1$ is trivial, hence we assume $N \ge 2$. Let $\Psi \in \K^{(N-1)\times (M-1)}$ be a Parseval
  frame for $\K^{N-1}$. We now extend $\Psi$ into an $N \times M$ matrix by first adding a $1\times (M-1)$ row
  vector with zeros as a new $N$th row, and then by adding $e_N \in \K^N$ as a new $M$th column. The picture
  is:
  \[ 
  \Phi =
  \begin{bmatrix}
    \ddots & \vdots & \iddots & 0 \\
    \cdots & \Psi  & \cdots & 0\\
    \iddots & \vdots & \ddots & 0 \\
    0 & 0 & 0 & 1
  \end{bmatrix}.
  \]
  The new frame $\Phi$ is a Parseval frame for $\K^N$ since $\Phi \Phi^\ast=\id[N]$. Moreover, it is prime
  since any tight frame $\Phi_1 \subset \Phi$ must contain the $M$th column in order to span $\K^N$.
\end{proof}

\begin{remark}
  For certain values of $N,M$, we can extend Proposition~\ref{thm:partial-existence} to show existence of
  prime \emph{unit-norm} tight frames.  The case $N\leq M< 2N$ follows from
  Remark~\ref{lem:atleast-2-times-dimension}. If $M \ge 2N$ and $M$ is prime, then we take the $M \times M$
  DFT matrix and choose any $N$ rows to be our frame $\Phi$. Since no proper subset of the (primitive) $M${th}
  root of unity sum to zero, there exists no way to divide $\Phi$ into two FUNTFs. Hence, $\Phi$ is a prime
  FUNTF.
\end{remark}

It is easy to build divisible tight frames since the union of any two tight frames is divisible. However, the
existence of prime tight frames allows us to prove that ``most'' tight frame are in fact prime. To state this
result, we need to set some notations. For $A>0$ fixed, let $\mathcal{F}(N,M,A)$ be the set of all tight
frames with frame bound $A$, that is,
\[
\mathcal{F}(N,M,A)=\setprop{\Phi \in \textrm{Mat}(N \times M,\K)}{\Phi\Phi^\ast=A \id[N]}.
\]
Let
\[
\mathcal{P}(N,M,A)=\setprop{\Phi \in \mathcal{F}(N, M, A)}{ \Phi \text{ is prime}}.
\] 
When $A=1$, we simply denote these sets $\mathcal{F}(N,M)$ and $\mathcal{P}(N,M)$.  Note that $\mathcal{F}(N,
M)$ is the Stiefel manifold, see \cite{strawn-frame-varieties}, which is invariant under multiplication by $N
\times N$ unitary matrices from the left and by $M \times M$ unitaries from the right. Consequently, there
exists an invariant Haar probability measure $\mu$ on $\mathcal{F}(N,M)$. The results stated below hold for
any measure that is absolutely continuous with respect to $\mu$.

Following the setup in \cite[Section 3]{fullspark}, each of the entries in a matrix $\Phi \in \mathcal{F}(N,
M)$ is written in terms of its real and imaginary parts: $\varphi_{k,\ell}=x_{k, \ell}+i y_{k, \ell}$ for all
$k, \ell$. In this setting we can consider $\mathcal{F}(N, M)$ as an real algebraic variety in $\R^{2M N}$. Since
$\mathcal{F}(N, M)$ is an irreducible variety, every non-empty Zariski-open subset of $\mathcal{F}(N, M)$ is
open and dense in $\mathcal{F}(N, M)$ in the (induced) standard topology~\cite{fullspark}. Moreover, the
complement of a non-empty Zariski-open subset is of $\mu$-measure zero. The following result says that being
prime within the set of tight frames is a \emph{very generic} property.

\begin{theorem}
\label{thm:prime-frames-dense}
Let $N,M\in \N$ and $A>0$. The set of prime $A$-tight frames $\mathcal{P}(N,M,A)$ is open and dense in the set
of all $A$-tight frames $\mathcal{F}(N,M,A)$ in the (induced) standard topology. Moreover, the complement
$\mathcal{F}(N,M,A) \setminus \mathcal{P}(N,M,A)$ is of measure zero in $\mu$.
\end{theorem}
  
\begin{proof}
  We can without loss of generality take $A=1$. Let $\Phi \in \mathcal{F}(N,M)$ and let $S$ be the power set
  of $\set{1,\dots,M}$. We use the notation $\Phi_I=[\varphi_i]_{i\in I}$ for $I \in S$. Now, $\Phi$ is
  divisible if and only if
  \[
  \Phi_I \Phi_I^{\phantom{I}\ast}= c \id[N]
  \] 
  for some $\emptyset \neq I \in S$ and $c >0$. These orthogonality conditions can be expressed as polynomial
  equations in $x_{k,\ell}$ and $y_{k,\ell}$ introduced above. Since $S$ is finite, the set of prime $1$-tight
  frames $\mathcal{P}(N,M)$ is a finite intersection of the complement of such sets in $\mathcal{F}(N,M)$. The
  sub-variety $\mathcal{P}(N,M)$ is therefore Zariski-open in the irreducible variety $\mathcal{F}(N,M)$. By
  Proposition~\ref{thm:partial-existence} the set $\mathcal{P}(N,M)$ is non-empty. Since $\mathcal{P}(N,M)$ is
  a non-empty Zariski-open set in an irreducible variety, the result follows.
\end{proof}

\begin{remark}
  Let $V$ be an $N \times M$ matrix with entries independently drawn at random from an absolutely continuous
  distribution with respect to the Lebesgue measure; a standard choice could be the Gaussian distribution of
  zero mean and unit variance. With probability one, $V$ is a frame and thus performing the Gram-Schmidt
  algorithm on the rows of $V$ leads to a tight frame.  We call such frames for random, tight frames. It can
  be shown from Theorem~\ref{thm:prime-frames-dense} that random tight frames are prime with probability one.
\end{remark}

By Theorem~\ref{thm:prime-frames-dense} we see that if $\Phi$ is a (divisible or prime) tight frame and
$\widetilde{\Phi}$ a random, arbitrarily small perturbation of $\Phi$ such that $\widetilde{\Phi}$ again is
tight, then $\widetilde{\Phi}$ is prime with probability one. From Theorem~\ref{thm:prime-frames-dense} we
also have the following density result.
 
\begin{corollary}
  Every tight frame is arbitrarily close to a prime tight frame.
\end{corollary}

\section{Classification of certain prime tight frames}
\label{sec:special-classes}  
In this section we characterize prime frames within three well-known families of FUNTFs: equiangular tight
frames, harmonic frames, and spectral tetris frames.

\subsection{Equiangular FUNTFs}
\label{sec:equiangular-funtfs}
A FUNTF $\Phi=\set{\varphi_i}_{i=1}^M$ is said to be \emph{equiangular} if
$\abssmall{\innerprod{\varphi_i}{\varphi_j}}=c$ for all $i,j=1,\dots,M$ with $i \neq j$ for some constant $c
\ge 0$, \cite{Strohmer:2003aa}. Equivalently, equiangular tight frames are sequences that achieve the
Welch bounds with equality \cite{Welch:1974aa}. We show that, when they exist, equiangular tight frames are
automatically prime tight frames. This necessary condition was derived by Xia, Zhou, and Giannikis
\cite{xzg05} for the special case of harmonic tight frames using the notion of difference sets.

\begin{theorem}[{$\!\!$\cite[Theorem 2.3]{Strohmer:2003aa}}]  
\label{thm:etf}
Suppose $\Phi=\{\varphi_i\}_{i=1}^{M}$ be a unit norm frame in $\K^N$. Then
\[
\max_{i\neq j}\abs{\innerprod{\varphi_i}{\varphi_j}}\geq \sqrt{\frac{M-N}{N(M-1)}},
\] 
and equality holds if and only if $\Phi$ is an equiangular tight frame. 
\end{theorem}

The existence of equiangular tight frames is still an open problem. However, as shown below, when they exist,
equiangular tight frames are also prime.

\begin{theorem}
\label{thm:irreetf}
Let $N\ge 2$. Equiangular FUNTFs of $M$ vectors in $\K^N$, when they exist, are prime.
\end{theorem}
\begin{proof}
  Assume that $\Phi=\{\varphi_i\}_{i=1}^{M}\subset \K^N$ is an equiangular tight frame. Assume towards a
  contradiction that $\Phi$ is $(M, p)$-divisible for some $N\leq p \leq M-N$. Then write $\Phi=\Phi_1 \cup
  \Phi_2$ where $\Phi_1$ and $\Phi_2$ are divisors of $\Phi$ of size $p$ and $M-p$,
  respectively. Consequently, we see that $\Phi$, $\Phi_1,$ and $\Phi_2$ are all equiangular tight
  frames. According to Theorem \ref{thm:etf},
  \[
  \abs{\innerprod{\varphi_i}{\varphi_j}}=\sqrt{\tfrac{M-N}{N(M-1)}}=\sqrt{\tfrac{p-N}{N(p-1)}}=\sqrt{\tfrac{M-p-N}{N(M-p-1)}}
  \]
  with $N\leq p \leq M-N$. But a series of easy calculations leads to a contradiction. Thus, $\Phi$ cannot be
  $(M, p)$-divisible for any $N\leq p \leq M-N$.
\end{proof}

Note that Grassmanian tight frames, see \cite{beko}, are not in general prime tight frames, \eg any
Grassmanian frame of four frame vectors in $\R^2$ is $(4,2)$-divisible.

\subsection{Harmonic frames}
\label{sec:harmonic-frames}
We now characterize all harmonic FUNTFs that are prime, and for those that are divisible we describe their
factors. For this, we recall that given $M\geq N$, a \emph{harmonic tight frame} (HTF) is obtained by keeping
and renormalizing the $N$ first coordinates from an $M\times M$ discrete Fourier transform, that is,
$\Phi=\{\varphi_{k}\}_{k=0}^{M-1}$ is a HTF, where the $(k+1)$th column of $\Phi$ is given by
\[ 
\varphi_{k}
=\sqrt{\frac{s}{N}}\begin{pmatrix} 1 \\ \gamma_{M}^{k}\\  \gamma_{M}^{2k}\\ \vdots \\
  \gamma_{M}^{(N-1)k}\end{pmatrix}
=\sqrt{\frac{s}{N}}\begin{pmatrix}\omega_{1}^{k  }\\ \omega_{2}^{k}\\
  \omega_{3}^{k}\\ \vdots \\
  \omega_{N}^{k}\end{pmatrix} \in \C^N,
\] 
where $s>0$, $\gamma_M:=\exp{(2 \pi i /M)}$ is the $M$th root of unity and $\omega_n:=\gamma_M^{n-1}=\myexp{2
  \pi i(n-1) /M}$ for $n=1,\dots,N$. When there is no risk of confusion we shall simply write $\gamma$ for
$\gamma_M= \exp{(2\pi i /M)}$. We denote the obtained HTF by $\HTF(N,M,s)$, and we see that this tight frame
has frame bound $A=sM/N$ and frame vector norms $\norm{\varphi_k}=s$. Hence, for $s=1$ we have a unit-norm,
$M/N$-tight frame, while we for $s=N/M$ have a Parseval frame. When nothing else is mentioned we assume $s=1$
for simplicity.

Let us fix some notations and assumptions for this section.  We will always assume $N\ge 2$ since, if $N=1$,
any HTF is $(M,p)$-divisible for all $p$.  Throughout this section we will denote the index set
$\set{1,2,\dots,M}$ by $I$, and we will let $J_1$ denote a subset of $I$ and put $J_2:=I \setminus J_1$.  If
$d$ is a divisor of $M$, we define index sets $I(d,q) \subset I=\{1, 2, \hdots M\}$ as
follows 
\begin{equation}
  \label{eq:divisor-index-sets}
  I(d,q)=\setprop{k\frac{M}{d}+q}{k=0,1,\dots, d-1}, \qquad q=1,\hdots,M/d.
\end{equation}
These index sets are a disjoint partition of $I=\set{1,2,\dots,M}$ for a fixed divisor $d$. For any $n=1,
2,\hdots,N-1$, the set $\setprop{(\gamma^{n})^{m}}{m\in I(d,1)}$ is a subgroup of
$\setprop{(\gamma^{n})^{m}}{m\in I}$ in the circle group, and $\setprop{(\gamma^{n})^{m}}{m\in I(d,q)}$ is a
coset.  Furthermore, we assume the following prime factorization $M=p_1^{\alpha_1}p_{2}^{\alpha_2}\cdots
p_{r}^{\alpha_{r}}$ with $\alpha_i \in \N$ and $p_i$ prime and ordered such that $p_i>p_{i+1}$.
  
Let $\Phi = \{\varphi_{k}\}_{k=1}^{M} \subset \C^N$ be a HTF with index set $I =\set{1,2,\dots,M}$. Since
$\Phi=\{\varphi_{k}\}_{k \in I}$ is a tight frame, the rows of $\Phi$ are equal-norm and orthogonal.  In
particular, we have
\begin{equation}
  0= \innerprodbig{\varphi^n}{\varphi^{n'}} = \sum_{m=1}^M (\gamma^{(n-1)-(n'-1)})^{m-1} =\sum_{m\in I} (\gamma^{n-n'})^{m-1}  
  \quad \text{for $n\neq n' \in \{1, 2, \hdots, N\},$}\label{eq:htf-row-ortho}
\end{equation}
where $\varphi^j$ denotes the $j$th row of $\Phi$.  Now, $\Phi$ is divisible exactly when it is possible to
split the sum over $m \in I$ into two sums, each summing to zero, for each $n\neq n'$. Here we have used that
the norms of the $N$ rows of any sub-collection of $\Phi$ are automatically equal since the entries of $\Phi$
are equal in modulus. Therefore it is only necessary to consider the row-orthogonality requirement of
potential divisors. We will use the following result repeatedly.
\begin{lemma}
  \label{lem:divisible-iff-vanishing-sums}
  Let $N,M \ge 2$ be given, and let $\Phi = \{\varphi_{k}\}_{k=1}^{M} \subset \C^N$ be a HTF. Then $\Phi$ is
  $(M,p)$-divisible if and only if there exists $J_1 \subset I=\set{1,\dots,M}$ with $\card{J_1}=p$ such that
  \begin{equation}
    0 =\sum_{m\in J_1} (\gamma^{n-1})^{m-1}  
    \qquad \text{for $n \in \{2, \hdots, N\}$}\label{eq:htf-row-ortho-to-first}.
  \end{equation}
\end{lemma}
\begin{proof}
  Let $\Phi_1 := \Phi_{J_1}=\set{\varphi_k}_{k \in J_1}$, and let $\varphi_1^j$ be the $j$th row of
  $\Phi_1$. Note that the equations in~\eqref{eq:htf-row-ortho-to-first} are equivalent to the statement that
  $\innerprods{\varphi_1^1}{\varphi_1^n}=0$ for $n=2,\dots,N$. Assume \eqref{eq:htf-row-ortho-to-first}
  holds. Then, since $\gamma^{-k}=\gamma^{M-k}$, we see that
\begin{equation*}
   0 =\sum_{m\in J_1} (\gamma^{n-n'})^{m-1}  
 \end{equation*}
 holds for all $n, n' \in \{2, \hdots, N\}$ with $n \neq n'$.  The last statement is equivalent to
 $\innerprods{\varphi_1^n}{\varphi_1^{n'}}=0$ for all $n \neq n' \in \set{2, \hdots, N}$ which in matrix
 notation becomes $\Phi_1 \Phi_1^\ast = c \id[n]$.  Therefore, equation~\eqref{eq:htf-row-ortho-to-first}
 implies that $\Phi_1$ is a tight frame. The converse implication follows easily from the above.
\end{proof}

Suppose $\Phi$ is $(M,p)$-divisible. By taking $n=2$ in \eqref{eq:htf-row-ortho-to-first} we see that the
sub-sum over $J_1$ must be the sum of $p$ $M$th roots of unity, and, of course, the second sum over $J_2$ must
be the sum of $(M-p)$ $M$th roots of unity. This is an example of vanishing sums of roots of unity
\cite{lamle00}. When one of the vanishing sub-sums contains $p$ distinct $M$th roots of unity, one says that
$M$ is \emph{$p$-balancing} \cite{sivek10}. In particular, we shall use the following result proved in
\cite{sivek10}.
  
\begin{theorem}[{$\!$\cite[Theorem 2]{sivek10}}]
\label{thm:sivek}
Let $M=p_1^{\alpha_1}p_{2}^{\alpha_2}\cdots p_{r}^{\alpha_{r}}$ with each $p_i$ prime and each
$\alpha_i>0$. Then $M$ is $k$-balancing if and only if both $k$ and $M-k$ are in $\N_0 p_1 + \N_0 p_2 + \hdots
+ \N_0 p_r$, that is, both $k$ and $M-k$ are linear combination of $p_i$ with nonnegative integer
coefficients.
\end{theorem}

By Lemma~\ref{lem:divisible-iff-vanishing-sums}, we immediately have that, for $N=2$, a HTF is
$(M,p)$-divisible if and only if $M$ is $p$-balancing. More precisely, we
have: 

\begin{corollary}
  \label{thm:harmonic-divisors-N2}
  Let $M\geq 2$ be given. Suppose $\Phi = \{\varphi_{k}\}_{k=1}^{M} \subset \C^2$ is a HTF. Then $\Phi$ is
  prime if and only if $M$ is a prime integer. Furthermore, if $M$ is not prime, then $\Phi$ is
  $(M,d)$-divisible for each $2\leq d \leq M-2$ for which both $d$ and $M-d$ are in $\N_0 p_1 + \N_0 p_2 +
  \hdots + \N_0 p_r$, in particular, $\Phi$ is $(M,d)$-divisible for every divisor $d$ of $M$.
\end{corollary}

\begin{proof}
  We will use the fact that the frame $\HTF(2,M,s)$ is prime if and only if $M$ is not $d$-balancing for any
  $2 \le d \le M-2$.

  Assume first that $\Phi$ is prime, that is, that $M$ is not $d$-balancing for any $2 \le d \le M-2$. Towards
  a contradiction assume further that $M$ is not prime so that $M=p_1^{\alpha_1}p_{2}^{\alpha_2}\cdots
  p_{r}^{\alpha_{r}}$ with $r>1$. Then $M=p_1b$ with $b = p_1^{\alpha_1-1}p_{2}^{\alpha_2}\cdots
  p_{r}^{\alpha_{r}}\ge 2$. For $d=p_1$ we have
  \[
  M-d = p_1 b - p_1 = p_1(b-1),
  \] 
  but this contradicts the fact that $M$ is not $d$-balancing for any $2 \le d \le M-2$.

  For the opposite implication, we observe that the prime factorization of $M=p_1$ is trivial when $M$ itself
  is prime. Since there is no divisor $0<d<M$ so that $d,M-d \in M\N_0$, we see that $M$ is not $d$-balancing
  for any $N\leq d \leq M-N$. Thus, we have proved the bi-implication part of the theorem.

  The ``furthermore'' statement follows immediately from the above and Theorem~\ref{thm:sivek}.
\end{proof}

In general, when $N>2$, the characterization of prime HTFs is more involved. Indeed, we now have multiple rows
consisting of $M${th} roots of unity, and we must insure, for $(M,d)$-divisibility, not only that each of
these row is $d$-balancing, but also that the subset of $d$ roots that sum to zero in each of these rows,
comes from the same columns.

In order to formally state this we need to define a few sets of integers. Using the notation fixed above, we
define

\begin{align}
  D_{M, N}&=\setpropbig{d\in \set{N,\dots, M-N}}{d \text{ is a divisor of } M},\label{eq:def-D_MN} \\
  P_{M,N}&=D_{M,N} \setminus \setprop{d \in D_{M,N}}{\exists c \in
    D_{M,N} \text{ such that $c$ is a divisor of $d$}}, \label{eq:def-P_MN}
\end{align}
and
\begin{align*} 
  S_{M,N} =\setpropbig{s\in \set{N,\dots, M-N}}{s=\sum_{k=1}^{K}a_k q_k, M-s=\sum_{k=1}^{K}b_k q_k, \text{
      where } a_k, b_k \in \N_0, q_k \in P_{M,N} },  
\end{align*}
where $K=\card{P_{M,N}}$.  Note that $s\in S_{M,N}$ if and only if $M-s \in S_{M, N}$.  It is also clear
that $$P_{M,N} \subset D_{M,N} \subset S_{M,N},$$ and that $D_{M,N}$ is empty for any $N\in \N$ if $M$ is
prime or if $M < 2N$. Note that if $M\geq 2N$, the condition $d\leq M-N$ in the definition of $D_{M, N}$ is
redundant as no divisor of $M$ can be greater than $M-N$. Moreover, the set $S_{M,N}$ is empty, if and only if
$D_{M,N}$ is empty. The significance of these sets will become evident in Theorems~\ref{harmonic-prime1} and
\ref{thm:harmonic-prime-and-divisors2} and Corollary~\ref{thm:harmonic-factors} below, but we mention here
that $P_{M,N}$ will determine the cardinality of the prime factors and $S_{M,N}$ the cardinality of every
possible divisor of the HTF.  In the following example we calculate these sets for various $M,N \in \N$.
\begin{example}
\label{exa:D-P-S-sets}
  \begin{compactenum}[(a)]
  \item For prime $M\in \Z$ and any $N \in \N$, we have $P_{M,N}=D_{M,N}=S_{M,N}=\emptyset$.
  \item For $M=9$, $N=2$ or $N=3$: $D_{M,N}=P_{M,N}=\set{3}$, $S_{M,N}=\set{3,6}$.
  \item For $M=9$, $N\ge 4$: $D_{M,N}=P_{M,N}=S_{M,N}=\emptyset$.
  \item For $M=10$, $N=2$: $D_{M,N}=P_{M,N}=\set{2,5}$, $S_{M,N}=\set{2,4,5,6,8}$.
  \item For $M=10$, $N=3, 4$ or $5$: $D_{M,N}=P_{M,N}=S_{M,N}=\set{5}$.
  \item For $M=24$, $N=2$: $D_{M,N}=\set{2,3,4,6,8,12}$, $P_{M,N}=\set{2,3}$, $S_{M,N}=\set{2,3,\dots,22}$.
  \item For $M=24$, $N=3$: $D_{M,N}=\set{3,4,6,8,12}$, $P_{M,N}=\set{3,4}$, $S_{M,N}=\set{3,\dots,21}\setminus \set{5,19}$. 
  \item For $M=24$, $N=4$: $D_{M,N}=\set{4,6,8,12}$, $P_{M,N}=\set{4,6}$, $S_{M,N}=\set{4,6,\dots,20}$.
  \end{compactenum}
\end{example}

Given $d \in \N$, the following well-known fact will be used repeatedly:  
\begin{equation}
  \label{eq:ross-hewitt}
\frac{1}{d}  \sum_{k =0}^{d-1} \mathrm{e}^{2\pi i \tfrac{km}{d}} =
  \begin{cases}
    1 & m \in d\Z ,\\
    0 & m \in \Z \setminus d\Z.
  \end{cases}
\end{equation}

We now state and prove the following result characterizing all HTFs that are prime. 
  
\begin{theorem}\label{harmonic-prime1}
  Let $N,M \in \N$ and $s>0$ be given, and let $\Phi = \{\varphi_{k}\}_{k=1}^{M}=\HTF(N,M,s) \subset
  \C^N$. Then $\Phi$ is prime if and only if $D_{M, N}=\emptyset.$ In particular, if $M$ is prime, then $\Phi$
  is prime.
\end{theorem}

\begin{proof}
  Suppose that $\Phi$ is prime and let us prove that $D_{M, N}=\emptyset$. Assume by way of contradiction that there
  exists $ d \in D_{M, N} \neq \emptyset$, \ie $d\ge N$ is a divisor of $M$. Take $J_1 = I(d,1)$, where $I(d,1)$ is
  defined in \eqref{eq:divisor-index-sets}. Then, for $n \in \{2, \hdots, d\}$,
  \begin{align*}
    \sum_{m\in J_1} (\gamma_M^{n-1})^{m-1}= \sum_{k=0}^{d-1} \myexp{\tfrac{2\pi i}{M} (n-1) (k \tfrac{M}{d})}
    = \sum_{k=0}^{d-1} \myexp{2\pi i \tfrac{k(n-1)}{d}} = 0,
  \end{align*}
  where the last equality follows from~\eqref{eq:ross-hewitt}. Since $d \ge N$, we see by
  Lemma~\ref{lem:divisible-iff-vanishing-sums} that $\Phi$ is $(M,d)$-divisible with divisor
  $\Phi_1:=\Phi_{J_1}$ which is a contradiction.

  We now prove that if $\Phi$ is not prime, then $D_{M, N}\neq \emptyset.$ Thus assume that $\Phi$ is $(M,
  d)$-divisible for some $d$ such that $N\leq d \leq M-N$. Let $\Phi=\Phi_1 \cup \Phi_2$ be divisors of
  $\Phi$, where $\Phi_1=\Phi_{J_1}$ and $\Phi_2=\Phi_{J_2}$ with $\card{J_1}=d$ and $\card{J_2}=M-d$.  By
  Lemma~\ref{lem:divisible-iff-vanishing-sums}, we see that our assumption is equivalent to assuming the
  existence of a index set $J_1 \subset \set{1,\dots,M}$ of cardinality $d$ such that
  \begin{equation}
    \sum_{m\in J_1} (\gamma_M^{n-1})^{m-1}  = \sum_{m\in J_1^n} \gamma_M^{m-1}  = 0\label{eq:Phi1-divisible-sum-on-J1} 
  \end{equation}
  holds for each $n=2,\dots,N$, where $J_1^n:=(n-1)J_1 \mod M$. In particular, for $n=2$, this means that $M$
  is $d$-balancing, hence $d=\sum_{k=1}^r a_k p_k$ and $M-d=\sum_{k=1}^r b_k p_k$ for $a_k,b_k \in \N_0$. Let
  $R_1 \subset \set{1,\dots,r}$ be the indices $k$ for which $a_k\neq 0$. We see that the index set $J_1$ has
  the following form
  \[
  J_1 = \bigcup_{k \in R_1} \bigcup_{q_k \in Q_k} I(p_k,q_k) = \bigcup_{k \in R_1} I_k, \qquad \text{where
  }I_k:=\bigcup_{q_k \in Q_k} I(p_k,q_k),
  \]
  for some $Q_k \subset \set{1,\dots,M/p_k}$ with $\card{Q_k}=a_k$ for each $k \in R_1 $.  We can assume
  without loss of generality that, for $i,j\in R_1$,
  \begin{equation}
    n_i p_i \neq n_j p_j \qquad \text{for each $n_i=1,\dots,a_i$ and $n_j=1,\dots,a_j$,}\label{eq:assump-on-prime}
  \end{equation}
  whenever $i\neq j$.

  We need to prove that $D_{M, N}\neq \emptyset$. It follows from $d=\sum_{k \in R_1} a_k p_k$ that, if
  $p_{k_{0}}\geq N$ for some $k_0 \in R_1$, then $p_{k_{0}}\in D_{M,N}\neq \emptyset$, which concludes the
  proof. If however, $p_k<N$ for each $k \in R_1$, we claim that $a_{k}p_{k}\geq N$ for each $k \in R_1$. To
  show this claim, let $k_0 \in R_1$ be fixed and assume $p_k<N$ for all $k \in R_1$. Since $p_k<N$ is prime,
  it follows that
  \[
  \sum_{m\in I_{k}} (\gamma_M^{(n-1)})^{m-1} = 0
  \]
  for $n=2,\dots,p_k,p_k+2,\dots$ for $k \in R_1$.  We remark that $p_{k_0}$ is not a $n_j$-multiple of any of
  the other prime numbers $p_j$ for $n_j=1,\dots,a_j$ by equation~(\ref{eq:assump-on-prime}) which, in turn,
  implies that
  \[
  \sum_{m\in J_1\setminus I_{k_0}} (\gamma_M^{p_{k_0}})^{m-1} = 0
  \]
  for $n=p_{k_{0}}+1$.  Hence, it follows by the fact that $\Phi_1$ is a tight frame, that
  \begin{equation}
    \sum_{m\in I_{k_0}} (\gamma_M^{p_{k_0}})^{m-1}  = p_k \sum_{q \in Q_{k_0}} (\gamma_M^{p_{k_0}})^{q-1} =0.\label{eq:d1}
  \end{equation}
  The last equality is only possible if the index set $Q_{k_0}$ is of a certain size, in particular,
  $\card{Q_{k_0}}\ge 2$ must be a divisor of $M/p_{k_0}$.  If $M/p_{k_0}$ is prime, then
  $a_{k_0}=\card{Q_{k_0}}=M/p_{k_{0}}$ which is impossible. Hence, $M/p_{k_0}$ is not prime. Let $d_1$ be the
  smallest divisor of $M/p_{k_0}$. By (\ref{eq:d1}) we then see that $\card{Q_{k_0}} = n d_1$ for some $n \in
  \N$ since $Q_{k_0}$ must be a union of $n$ sets, say $D^1_\ell$, $\ell=1,\dots,n$, of cardinality $d_1$. If
  $d_1 p_{k_0}>N$, we are done. If not, we consider row $n=d_1p_{k_{0}}+1$ in
  \eqref{eq:Phi1-divisible-sum-on-J1}. Repeating the argument above leads to
  \begin{equation}
    \sum_{q \in Q_{k_0}} (\gamma_M^{d_1
      p_{k_0}})^{q-1} =0.\label{eq:d2}
  \end{equation}
  Note that $(\gamma_M^{d_1 p_{k_0}})^{q-1}$ is the \emph{same} $M$th root for all $q \in D^1_\ell$ when
  $\ell$ is fixed. If $M/(d_1 p_{k_0})$ is prime, then, by (\ref{eq:d1}) and (\ref{eq:d2}),
  $\card{Q_{k_0}}=M/p_{k_{0}}$ which is impossible. Hence, $M/(d_1 p_{k_0})$ is not prime, and we let $d_2$ be
  the smallest divisor of $M/(d_1 p_{k_0})$. By (\ref{eq:d1}) and (\ref{eq:d2}) we then see that
  $\card{Q_{k_0}} = n d_1 d_2$ for some $n \in \N$. We can continue the argument which proves the claim.

  By the proof of the claim, we also see that $a_{k_0}p_{k_0}$ is a divisor of $M$. Since we just proved
  $a_{k_0}p_{k_0}>N$, we arrive at the conclusion $a_{k_0}p_{k_0} \in D_{M,N}\neq \emptyset$.

  The proof of the last statement in Theorem~\ref{harmonic-prime1} is an easy consequence of the fact that if
  $M$ is prime, then $D_{M,N}=\emptyset.$
\end{proof}

\begin{remark}
\begin{compactenum}[(a)]
\item The last part of the above result states that a HTF $\Phi = \{\varphi_{k}\}_{k=1}^{M} \subset \C^N$ with
  $M\geq 2N$ is prime if $M$ is prime.  The converse of this is not true. Indeed, consider the HTF with $M=9$
  and $N=4$, see Example~\ref{exa:D-P-S-sets}(c). This frame is prime, but $M$ is not.
\item Theorem~\ref{thm:sivek} classifies all $k$ for which $M$ is $k$-balancing, but it does not give how to
  choose the $k$ distinct roots out of the $M$ roots of unity. What the proof of Theorem~\ref{harmonic-prime1}
  shows is that for a HTF to be $(M, d)$-divisible, in addition to having a vanishing sub-sum of $d$ distinct
  $M$th roots of unity, we must also ensure that the sub-sum of every $n$th power of \emph{these} roots of
  unity vanishes for $n=2,\dots, N-1$. We refer to Corollary~\ref{thm:harmonic-factors} below for a statement
  on how to choose the $d$ distinct roots of unity so that all these sub-sums vanishes.
\end{compactenum} 
\end{remark}

In case the HTF is divisible \ie $D_{M, N}\neq \emptyset$, the following result tells us how to factor it into
prime divisors. The proof follows from the proof of Theorem~\ref{harmonic-prime1}.

\begin{corollary}
  \label{thm:harmonic-factors}
  Let $N,M \in \N$ and $s>0$ be given, and let $\Phi = \{\varphi_{k}\}_{k=1}^{M}=\HTF(N,M,s) \subset
  \C^N$. Suppose that $D_{M,N}\neq \emptyset$ and $d<M$.  Then the following assertions are equivalent:
  \begin{compactenum}[(i)]
  \item $d \in P_{M,N}$,
  \item $\Phi$ factors into $M/d$ prime FUNTFs each of cardinality $d$, that is,
    \[
    \Phi=\bigcup_{i=1}^{M/d} \Phi_i
    \]
    with $\Phi_i$ being prime and $\card{\Phi_i}=d$.
  \end{compactenum}
  Furthermore, one of the factors from assertion~(ii), say $\Phi_1$, is the HTF with $d$ frame vectors is
  $\C^N$, that is, $\Phi_1=\HTF(N,d,s)$.  Let $U=\diag{\gamma^0,\gamma^1, \dots, \gamma^{N-1}} \in
  \text{U}(N)$, where $\gamma=\exp{(2\pi i/M)}$. The other factors are then given as $\Phi_i=U^{i-1}\Phi_1$
  for $i=1,\dots,M/d$.
\end{corollary}

By Corollary~\ref{thm:harmonic-factors} we know exactly how the prime ``building blocks'' of a divisible HTF
look, hence we can use this structure to build larger divisors of the HTF. Suppose $\Phi$ is a HTF and $d \in
D_{M, N}\neq \emptyset$. Then $\Phi$ is $(M,d)$-divisible, and, moreover, $\set{\varphi_i}_{i \in I(d,q)}$ is
a tight frame for any $q=1,\dots,M/d$. Now, we can combine these $M/d$ tight frames of cardinality $d$ into
tight frames of cardinality $d,2d,3d, \dots, M$. Hence, $\Phi$ is also $(M,nd)$-divisible for
$n=1,\dots,M/d-1$. Assume further that $M$ has another divisor which are also greater than $N$, say $\tilde d
\in D_{M,N}$. We can then combine unions of $\set{\varphi_i}_{i \in I(d,q)}$ with unions of
$\set{\varphi_i}_{i \in I(\tilde{d},\tilde{q})}$, where $q=1,\dots,M/d$ and $\tilde{q}=1,\dots,M/\tilde{d}$,
respectively, as long as the index sets are disjoint. Hence, to combine tight frames from different divisors,
we only need to make sure that the same frame element $\varphi_i$ does not appear in both frames. We make
these observations precise in the following result.

\begin{theorem}
  \label{thm:harmonic-prime-and-divisors2}
  Let $M\geq N\geq 2$, and $\Phi = \{\varphi_{k}\}_{k=1}^{M} \subset \C^N$ be a HTF. If $D_{M,N}\neq
  \emptyset$, then $\Phi$ is $(M,s)$-divisible for each $s \in S_{M,N}$.
\end{theorem}

\begin{proof}
  If $ D_{M, N} \neq \emptyset$, then $\Phi$ is $(M, d)$-divisible for any divisor $d$ of $M$ such that $N\leq
  d \leq M-N$. We now show that $\Phi$, in fact, is $(M, s)$-divisible for each $s \in S_{M, N}$. By symmetry
  of the set $S_{M,N}$, it suffices to show that $\Phi$ is $(M,s)$-divisible for $s \in S_{M, N}$ with $s \le
  M/2$. For $s \in S_{M,N}$ we have $s=\sum_{k=1}^{K}a_k q_k$ with $a_k \in \N_0, q_k \in P_{M, N}$ and
  $K=\card{P_{M,N}}$. For each $k=1, 2, \hdots k$ construct an $N \times q_{k}$ matrix $\Phi_{\ell}^{(p_{k})}$
  based on the first $q_k${th} roots of unity. By shifting this matrix $a_k -1$ times, we will have a
  collection of $a_k$ such matrices. Next define $$\Phi_{k}=\begin{bmatrix} \Phi^{(q_{k})}_1 &
    \Phi^{(q_{k})}_2 & \cdots & \Phi^{(q_{k})}_{a_{k}}\end{bmatrix}.$$ Now $\Phi_{k}$ is an $N \times
  a_{k}q_k$ matrix whose rows are mutually orthogonal. We then 
  obtain an $N\times d$ matrix $$\widetilde{\Phi_{1}}=\begin{bmatrix}\Phi_1 & \Phi_2 & \cdots &
    \Phi_r\end{bmatrix}$$ which is a FUNTF. Hence, $\Phi$ is $(M, s)$-divisible.
\end{proof}

Casazza and Kova\v{c}evi\'{c} introduced the notion of generalized harmonic frames in~\cite{cako03} and showed
that these frames are unitarily equivalent to HTFs. Consequently, by Proposition~\ref{thm:unitary-equi}, the
results obtained in this section for HTFs extend to a classification of prime and divisible generalized HTFs.

\subsection{Computational aspects of harmonic tight frames}
\label{sec:comp-aspects-HTFs}

For harmonic tight frames the prime building blocks are exactly described by
Corollary~\ref{thm:harmonic-factors}.  We wish to suggest a strategy that can be used to design efficient
numerical tools for signal processing with divisible HTFs. Recall that, in analyzing a signal $x \in \C^N$
with any type of frame, one needs to compute $c=\set{\innerprod{x}{\varphi_i}}_{i=1}^M$. A na\"ive way of
computing the analysis step for divisible harmonic tight frames would be to zero pad $x\in \C^N$ into a vector
$\hat x \in \C^M$ and then compute a full FFT of $\hat x$ of size $M$. However,
Corollary~\ref{thm:harmonic-factors} suggests a more effective analysis (and synthesis) process.  Let
$\Phi=\set{\varphi_i}_{i=1}^M$ be a divisible HTF.  Suppose $p \in P_{M,N}$, where the set $P_{M,N}$ is
defined by ~\eqref{eq:def-P_MN}. The indices $I_q:=I(p,q)=\setpropsmall{kM/p+q}{k=0,1,\dots, p-1}$ for each
$q=1,\hdots,M/p$ are then the index sets of the $M/p$ prime factors $\Phi_{I_q}:=\set{\varphi_i}_{i\in
  I_q}$. By Corollary~\ref{thm:harmonic-factors} these prime factors are closely related, in fact,
$\Phi_{I_q}=U^{q-1}\Phi_{I_1}$, where $U=\diag{\gamma^0,\gamma^1, \dots, \gamma^{N-1}} \in \text{U}(N)$ is a
unitary, diagonal matrix and $\gamma$ the $M$th root of unity. Note that $\Phi_{I_1}$ is first $N$ rows of a
$p \times p$ discrete Fourier matrix (up to scaling).  Therefore, the analysis process of computing
$\set{\innerprod{x}{\varphi_i}}_{i=1}^M$ for some signal $x \in \C^N$ can be performed by first computing
$y_q:=(U^\ast)^{q-1}x \in \C^N$ for each $q=1,\dots,M/p$, then computing $c_{I_q}:=\Phi_{I_1}^{\,\ast} y_q \in
\C^p$ for $q=1,\dots,M/p$, and finally combining $c_{I_q}$ for each $q=1,\dots,M/p$ into $c$. The first step
$(U^\ast)^{q-1}x$ is fast and stable since $U$ a is diagonal unitary, the second step is $M/p$ FFTs of size
only $p$, and the last step is simply a rearrangement of the coefficients.  The number of operations needed in
this decomposition strategy is of the order of $M\log_{2}p$ which should be compared to $O(M \log_2 M)$
operations for the na\"ive strategy. However, more importantly, this decomposition can be implemented as a
parallel algorithm for each $q=1,\dots,M/p$, where we only have $O(p\log_{2}p)$ operations on each processor;
this will lead to a significant speed-up in the analysis step in multi-core and multiprocessing systems. A
similar speed-up procedure can be used for the synthesis process.

In addition, recall that the worst-case \emph{coherence} of a FUNTF $\Phi$ is given by
\[
\mu_{\Phi}=\max_{k\neq \ell}\abs{\innerprod{\varphi_k}{\varphi_\ell}}.
\]
Given $M \ge N$, the coherence of $\Phi=HTF(N,M,1)$ is easily computed as
\[ 
\mu_\Phi= \frac{1}{N}\max_{k\neq \ell} \abs{\sum_{n=0}^{N-1}\gamma_{M}^{n(k-\ell)}}= \frac{1}{N} \frac{\sin
  (\pi N/M)}{\sin (\pi /M)}=:\mu_{N,M}
\] 
since the maximum is obtained for $\abs{k-\ell}=1$.  If $\Phi=HTF(N,M,1)$ is divisible, then for each $p\in
P_{M, N}$, each of its $M/p$ prime HTFs have the same coherence $\mu_1$, \ie 
\[ 
\mu_1=\mu_{N,p}=\frac{1}{N} \frac{\sin (\pi N/p)}{\sin (\pi /p)},
\] 
which clearly satisfies $\mu_1 \leq \mu_{\Phi}$. In fact, we always have $\mu_1<\mu_{\Phi}$ if $P_{M,N}\neq
\emptyset$. Clearly, as the redundancy grows, \ie $M\to \infty$ with $N$ fixed, we see that $\mu_\Phi \nearrow
1$.  Thus decomposing divisible HTFs in their prime factors, results in these divisors having smaller
coherence, \eg if $p=dN$, we see that $\mu_1 \searrow d \sin{(\pi/d)}/\pi$ as $N \to \infty$. Moreover, for
the typical range $N \le p \le 2N$, we see that $\mu_1 \le 2/3$ for any $N \ge 3$.

\subsection{Spectral Tetris frames}
\label{sec:spectral-tetris}

Recently, Casazza, Fickus, Mixon, Wang, and Zhou~\cite{cfmwz11} introduced the \emph{spectral tetris} method
as a mean to construct FUNTFs for $\R^N$ for any given $N,M \in \N$ satisfying $M \ge 2N$. One of the key
features of this class of frames is that they are \emph{sparse} in the sense that the coefficient vector of
each frame element with respect to an orthonormal basis contains only few nonzero entries~\cite{cchkp}.  For
example, when $N=4$ and $M=11$, the spectral tetris construction yields the frame $\Phi=\set{\varphi_{i}}_{
  \in I}$ with $I = \set{1,2, \dots, 11}$:
\begin{equation}
 \Phi =
\begin{bmatrix}
  1 & 1 & \sqrt{\frac{3}{8}} & \sqrt{\frac{3}{8}} & 0 & 0 & 0 & 0 &
  0 & 0 & 0 \\
  0 & 0 & \sqrt{\frac{5}{8}} & -\sqrt{\frac{5}{8}} & 1 &
  \sqrt{\frac{1}{4}} & \sqrt{\frac{1}{4}} & 0 & 0 & 0 & 0 \\
  0 & 0 & 0 & 0 & 0 & \sqrt{\frac{3}{4}} & -\sqrt{\frac{3}{4}} & 1 &
  \sqrt{\frac{1}{8}} & \sqrt{\frac{1}{8}} & 0 \\
  0 & 0 & 0 & 0 & 0 & 0 & 0 & 0 & \sqrt{\frac{7}{8}} &
  -\sqrt{\frac{7}{8}} & 1
\end{bmatrix},\label{eq:STF-4-11}
\end{equation}
where the frame vectors of $\Phi$ appears as columns in the matrix. Observe that the first two vectors are
identical, which might be an undesirable for encoding because they lead to transform coefficients
$\innerprod{f}{\varphi_i}$ carrying no new information~\cite{Goyal:2001aa}.  Since $E_4=\set{\varphi_i}_{i\in
  I_1}$, $I_1=\{1,5,8,11\}$, is an orthonormal basis, we see that $\Phi_1=\Phi\setminus E_4$ is a tight frame
with redundancy reduced to $7/4$, and this can not be reduced further if tightness should be preserved.  The reduced tight frame $\Phi_2$ does
not have the artifact of $\Phi$ with respect to repetition of vectors. Moreover, it is a tight frame with
redundancy $M/N$ less than two, something that is not possible using spectral tetris algorithms from
\cite{cchkp,cfmwz11}. Hence, this reduced tight frame possesses additional desirable properties as compared to
the original tight frame, without loosing the sparsity of the original frame $\Phi$.

We prove below that all spectral tetris tight frames can be decomposed similarly, and we characterize all
prime spectral tetris frames. This characterization ultimately allows us to determine precisely when the
spectral tetris construction works for $M/N \le 2$ (Corollary~\ref{cor:spectr-tetris-redund-less-two}).

Before proving our main results we recall how the spectral tetris method works. Given any $N, M \in \N$
satisfying $M \ge 2N$, let $\lambda=M/N$. The method was developed in \cite{cfmwz11} and extended in
\cite{cchkp}. Here, we shall consider spectral tetris frames constructed from the algorithms in \cite{cfmwz11,
  cchkp} under the usual sparsity setup that the ``tetris blocks'' are of size $1 \times 1$ and $2 \times 2$.

We define $K=\setprop{k_n}{n=0,1,\dots,\gcd{\set{N,M}}}$, where $k_n =\frac{n N}{\gcd{\set{N,M}}}$.  For $N,M
\in \N$ such that $\lambda=M/N \ge 2$ the spectral tetris frame $\mathrm{STF}(N,M)$ in $\R^N$ is then given as
the $M$ columns of
\[
\begin{bmatrix}
   1 & \cdots & 1 & a_1 & a_1 & & & & & & &  & & & & & & & \\
   & & & b_1 & -b_1 & 1  & \cdots & 1  & a_2 & a_2 & & & & & & & & &  \\
   & & & & & & & & b_2 & -b_2 &  1    & \cdots  & & & & & & &\\
   & & & & & & & &  &  &  &   &  & &   & & & &\\
   & & & & & & & &  &   & &  & \cdots & 1  & a_{N-1} & a_{N-1} & & &   \\
     \multicolumn{3}{c}{\myunderbrace{\phantom{1 \quad \cdots \;\;
           1}}{m_1 \text{ times}}} &  & &
     \multicolumn{3}{c}{\myunderbrace{\phantom{1 \quad \cdots \;\;
           1}}{m_2 \text{ times}}} & & & & &  &  & b_{N-1} & -b_{N-1} & \multicolumn{3}{c}{\myunderbrace{1 \;\;\; \cdots \;\;\;
           1}{m_N \text{ times}}} 
\end{bmatrix},\medskip
\]
where $a_j:=\sqrt{\frac{r_j}{2}}$ and $b_j:=\sqrt{1-\frac{r_j}{2}}$, and $r_j \in \itvco{0}{1}$ and $m_j \in
\N_0$, $j=1,\dots,N$ are defined by
\begin{align}
  \lambda &= m_j + r_j, && \text{when $j-1\in K$} ,
  \label{eq:m-and-r-in-K}  \\
  \lambda &= (2-r_{j-1}) + m_j + r_j, && \text{otherwise}. \label{eq:m-and-r--not-in-K}
\end{align}
If $j\in K$, the $2\times 2$-block matrix $B_j=\left[
  \begin{smallmatrix}
    a_j & a_j \\ b_j & -b_j
  \end{smallmatrix} \right]$ is left out. Note that $r_j=0$ exactly when $j \in K$.

The following result characterizes prime spectral tetris tight frames.

\begin{theorem}
\label{thm:factorization-of-STF}
Let $N,M \in \N$ such that $\lambda:=M/N\ge 2$ be given.  Suppose $\Phi$ is a spectral tetris FUNTF of $M$
vectors in $\R^N$ associated with $\lambda$.  Then, either $\Phi$ is prime, or $\Phi$ factors as
\[ \Phi = \Phi_1 \cup \bigcup_{l=1}^L E_N,
\] 
where $\Phi_1$ is a prime FUNTF of $M-LN$ vectors and
$E_N=\set{e_j}_{j=1}^N$ is the standard orthonormal basis for $\R^N$. Moreover,
$\Phi$ is divisible exactly when the eigenvalue $\lambda$ of the frame
operator of $\Phi$ satisfy
\begin{equation}
 j\lambda - \floor{(j-1)\lambda} \ge 3 \qquad \text{for all $1<j<\frac{N}{\gcd{(N,M)}}$}.\label{eq:spec-tetris-cond}
 \end{equation}
In particular, $\Phi$ is divisible if $\lambda\ge 3-\frac{1}{N}$, \ie $M \ge 3N-1$. 
\end{theorem}
\begin{proof}
  Let $\setsmall{e_j}_{j=1}^N$ denote the standard orthonormal basis of $\R^N$, and let $\Phi$ be a spectral
  tetris frame. We will use the notation introduced above. We claim that $\Phi$ is divisible if and only if
  $e_j \in \Phi$ for all $j=1,\dots,N$. If $e_j \in \Phi$ for all $j=1,\dots,N$, then $\Phi\setminus
  \setsmall{e_j}_{j=1}^N$ and $ \setsmall{e_j}_{j=1}^N$ are tight frames hence $\Phi$ is divisible. This shows
  one direction of the claim. Now, assume that $e_{j_0} \notin \Phi$ for some $j_0=1,\dots,N$. In other words,
  we assume that $m_{j_0}=0$. We consider the two $2\times 2$ blocks $B_{j_0-1}$ and $B_{j_0}$. Let $i_0+l \in
  \set{1,\dots,M}$ for $l=0,1,2,3$ be the indices of the columns of $\Phi$ associated with $B_{j_0-1}$ and
  $B_{j_0}$. Let $\varphi^{j_0}$ denote the $j_0$th row of $\Phi$. Assume towards a contradiction that
  $\Phi=\Phi_1 \cup \Phi_2$ is divisible, where $\Phi_1=\set{\varphi_i}_{i\in I_1}$ and
  $\Phi_1=\set{\varphi_i}_{i \in I_2}$ are tight frames. The common support of the rows $\varphi^{j_0}$ and
  $\varphi^{j_0+1}$ is $\set{i_0+2,i_0+3}$. Therefore, owing to orthogonality requirements of the rows, the
  frame vectors $\varphi_{i_0+2}$ and $\varphi_{i_0+3}$ need to belong to the same divisor, say $\Phi_1$. This
  in turn means that $\varphi_{i_0}, \varphi_{i_0+1} \in \Phi_2$ since $\Phi_2$ otherwise cannot span
  $\R^N$. The square norm of the $j_0$th row of $\Phi_1$ is $a_{j_0}^2+a_{j_0}^2=r_{j_0}$, and the square norm
  of the $(j_0+1)$th row of $\Phi_1$ is at least $b_{j_0}^2+(-b_{j_0})^2=2-r_{j_0}$. Since $\Phi_1$ has equal
  row norm, this implies that $r_{j_0}\ge 2-r_{j_0}$, that is, $r_{j_0}\ge 1$ which contradicts
  $r_{j_0}\in\itvco{0}{1}$. Hence, $\Phi$ is prime. This completes the proof of the claim.

  In \cite{kytuniok2012} it is shown that $e_j \in \Phi$ for all $j=1,\dots, N$ if and only if
  \begin{equation}
    (j_0 -k_n) \lambda - \floor{(j_0-1-k_n)\lambda} \ge 3\label{eq:STF-cond-from-other-paper}
  \end{equation}
  for any $j_0$ such that $k_n+1 < j_0 < k_{n+1}$ for all $n=0,1,\dots, \gcd{\set{N,M}}$. However, if
  (\ref{eq:STF-cond-from-other-paper}) is satisfied for one $n=0,1,\dots, \gcd{\set{N,M}}$, then it is
  satisfied for all $n=0,1,\dots, \gcd{\set{N,M}}$. Hence, $e_j \in \Phi$ for all $j=1,\dots, N$ if and only
  if
  \begin{equation*}
    j_0 \lambda - \floor{(j_0-1)\lambda} \ge 3
  \end{equation*}
  for all $j_0$ satisfying $1=k_0+1<j_0 < k_1=N/\gcd{\set{N,M}}$.
\end{proof}

The prime FUNTF $\Phi_1$ from Theorem~\ref{thm:factorization-of-STF} is actually the output of the spectral
tetris algorithm with $M-LN$ vectors.  The frame bound of $\Phi_1$ is $\lambda_1=\frac{M-LN}{N}=\lambda-L$,
which might be strictly smaller than two.  It is not difficult to see that the spectral tetris construction
always works when $M \ge 2N-1$, but it is in general difficult to see for which $M$ and $N$ the construction
still works when $N< M < 2N-1$. The next result characterizes when this is indeed possible. We denote the
output of the spectral tetris algorithm $\STF(N,M)$ for all $N,M \in \N$. If the spectral tetris does not work
for the given $N$ and $M$, we set $\STF(N,M) =\emptyset$.

\begin{corollary}
\label{cor:spectr-tetris-redund-less-two}
Let $N, \tilde{M} \in \N$ be given such that $N < \tilde{M} < 2N$. Define ${M}=\tilde{M}+N$ and
$\lambda=M/N\ge 2$. Then $\STF(N,\tilde{M})$ is a FUNTF if and only if $\lambda, N$ and $M$ satisfies
(\ref{eq:spec-tetris-cond}).
\end{corollary}

We mention that Corollary~\ref{cor:spectr-tetris-redund-less-two} was subsequently and independently proved in
\cite{chkwz12}.

\begin{remark}
\begin{compactenum}[(a)]
\item Let $N,M \in \N$, and let $D$ be the common divisors of $N$ and $M$. By
  Corollary~\ref{cor:spectr-tetris-redund-less-two} we see that $\STF(N,M)$ produces a FUNTF if and only if
  $M\ge 2N-1$ or $M=2N-d$ for some $d\in D$. This result is somewhat disappointing since it implies that spectral tetris frames with redundancy below two only are realizable by taking unions of spectral tetris frames for subspaces, where the number of vectors for each subspace frame is two times the subspace dimension minus one. 
\item Consider the frame $\STF(4,11)$ from (\ref{eq:STF-4-11}). Since $\lambda=11/4, N=4$, and $M=11$
  satisfies (\ref{eq:spec-tetris-cond}), it follows from Corollary~\ref{cor:spectr-tetris-redund-less-two}
  that $\STF(4,\tilde{M})$ with $\tilde{M}=M-N=7$ produces a FUNTF. Moreover, we see that
  $\STF(4,11)=\STF(4,7) \union E_4$.
\end{compactenum}
\end{remark}

\section*{Acknowledgments}
The authors would like to thank Jameson Cahill for valuable discussions. K.~A.~Okoudjou was supported by ONR
grants: N000140910324 \& N000140910144, and by the Alexander von Humboldt foundation. He would also like to
express his gratitude to the Institute for Mathematics at the University of Osnabr\"uck for its hospitality
while part of this work was completed.

\end{document}